\documentclass[11pt,reqno]{amsart}
\usepackage{indentfirst, amssymb,amsmath,amsthm, mathrsfs,cite,graphicx,float}
\newtheorem*{theoA}{Theorem A}
\newtheorem*{theoB}{Theorem B}
\newtheorem*{theoC}{Theorem C}

\newtheorem{theo}{Theorem}[section]
\newtheorem{lem}{Lemma}[section]

\newtheorem{rem}{Remark}[section]

\newtheorem{ques}{Question}[section]

\newtheorem{defi}{Definition}

\newtheorem{open problem}{Open problem}[section]
\newcommand{\pa}{\partial}
\newcommand{\ol}{\overline}
\newcommand{\be}{\begin{equation}}
\newcommand{\ee}{\end{equation}}
\newcommand{\bs}{\begin{small}}
\newcommand{\es}{\end{small}}
\newcommand{\beas}{\begin{eqnarray*}}
\newcommand{\eeas}{\end{eqnarray*}}
\newcommand{\bea}{\begin{eqnarray}}
\newcommand{\eea}{\end{eqnarray}}
\renewcommand{\epsilon}{\varepsilon}
\numberwithin{equation}{section}
\begin{document}
\title[Landau-type theorems]{Landau-type theorems for certain bounded poly-analytic and reduced poly-analytic functions}
\author[V. Allu, R. Biswas, R. Mandal and H. Yanagihara]{Vasudevarao Allu, Raju Biswas, Rajib Mandal and Hiroshi Yanagihara}
\date{}
\address{Vasudevarao Allu, Department of Mathematics, School of Basic Science, Indian Institute of Technology Bhubaneswar, Bhubaneswar-752050, Odisha, India.}
\email{avrao@iitbbs.ac.in}
\address{Raju Biswas, Department of Mathematics, Raiganj University, Raiganj, West Bengal-733134, India.}
\email{rajubiswasjanu02@gmail.com}
\address{Rajib Mandal, Department of Mathematics, Raiganj University, Raiganj, West Bengal-733134, India.}
\email{rajibmathresearch@gmail.com}
\address{Hiroshi Yanagihara, Department of Applied Science, Faculty of Engineering, Yamaguchi University, Tokiwadai, Ube, 755-8611, Japan}
\email{hiroshi@yamaguchi-u.ac.jp}
\maketitle
\let\thefootnote\relax
\footnotetext{2020 Mathematics Subject Classification: 31A30, 30C99.}
\footnotetext{Key words and phrases: Landau-type theorem, Bloch theorem, poly-analytic function, univalent functions.}
\begin{abstract} 
The main aim of this paper is to establish several Landau-type theorems for certain bounded poly-analytic functions and reduced poly-analytic functions that generalize some previously established results.
\end{abstract}
\section{Introduction and Preliminaries}
A function $F$ defined in a domain $\Omega\subset\mathbb{C}$ is called a poly-analytic function of order $m$ if it is represented in the following form
\bea\label{e1} F(z)=\sum_{k=0}^{m-1}\ol{z}^k f_k(z),\eea
where all $f_k$ $(k=0,1,\cdots, m-1)$ are analytic functions in $\Omega$ and $\ol{z}$ is the conjugate complex variable to the variable $z=x+iy$ (see \cite[Chapter 1]{B1991} and \cite{B1997}). The function $f_k$ is called the 
$k$-th analytic component of the poly-analytic function $F$. In other words, a poly-analytic function is a polynomial in $\ol{z}$ with coefficients that are analytic functions in $z$. 
These functions were initially explored by G. V. Kolossov \cite{K1908} in the context of developing a general method for solving problems in plane elasticity.
The significant and productive investigations of G. V. Kolossov and N. I. Muskhelishvili and numerous subsequent contributors regarding the applications of these functions to elasticity are well-known (see \cite{M1968}).
A poly-analytic function of order $m=2$ is called bianalytic.
The necessary and sufficient condition for a continuous complex-valued function $F(z)$ to be poly-analytic of the order $m$ in some region $\Omega\subset\mathbb{C}$ is that 
\beas \frac{\pa^m }{\pa \ol{z}^m}F(z)=0\quad\text{for all}\quad z\in\Omega.\eeas
The class of poly-analytic functions has been the subject of study by various authors from a variety of perspectives. For further information, we refer to \cite{A2010,A2011,AB1988,V1999} and the references therein.
A function $F$ defined in a domain $\Omega\subset\mathbb{C}$ is called a reduced poly-analytic function of order $m$ if it is represented in the following form
\beas F(z)=\sum_{k=0}^{m-1}|z|^{2k} f_k(z),\eeas
where all $f_k$ $(k=0,1,\cdots, m-1)$ are analytic functions in $\Omega$. Any polynomial in $z$ and $|z|^2$ may serve as an example of a reduced poly-analytic function. 
For a complex-valued function $F$ in $\Omega$, its Jacobian $J_F(z)$ is given by $J_F(z)=|F_{z}(z)|^2-|F_{\ol{z}}(z)|^2$. The inverse function theorem and a result of Lewy 
\cite{L1936} demonstrates that a harmonic function $F$ is locally univalent in $\Omega$ if, and only if, its Jacobian $J_F(z)\not=0$ in $\Omega$. A harmonic
mapping $F$ is locally univalent and sense-preserving in $\Omega$ if, and only if, $J_F(z)>0$ in $\Omega$.
For a continuously differentiable function $F$, let
\beas \Lambda_F(z)&=&\max_{0\leq \theta\leq 2\pi}|e^{i\theta}F_z+e^{-i\theta}F_{\ol{z}}|=|F_z|+|F_{\ol{z}}|\quad\text{and}\\[2mm]
\lambda_F(z)&=&\min_{0\leq \theta\leq 2\pi}|e^{i\theta}F_z+e^{-i\theta}F_{\ol{z}}|=\left||F_z|-|F_{\ol{z}}|\right|.
\eeas
For $r>0$, let $\mathbb{D}_r=\{z\in\mathbb{C} : |z|<r\}$ denote the open disk about the origin with radius $r$. In particular, $\mathbb{D}_1:=\mathbb{D}$, the open unit disk in the complex plane $\mathbb{C}$. \\[2mm]
A domain $\Omega\subset\mathbb{C}$ is said to be starlike with respect to the origin if, and only if, the line segment joining 0 to every other point $w\in\Omega$ lies entirely in $\Omega$.
\begin{defi}\cite{M1980,BP2014} A continuously differentiable function $f$ on $\Omega$ is said to be fully starlike in $\Omega$ if it is sense-preserving, $f(0)=0$, $f(z)\not=0$ in $\Omega\setminus\{0\}$ and the curve $f(re^{it})$ is starlike with respect to the origin for each $0<r<1$. The last condition is same as  
\beas \frac{\pa}{\pa t} \arg f(re^{it})=\frac{\pa}{\pa t}\text{Im}\left( \log f(re^{it})\right)=\text{Re}\left(\frac{zf_z(z)-\ol{z}f_{\ol{z}}(z)}{f(z)}\right)>0\eeas
for all $z=r e^{it}$ and $r\in(0,1)$.
\end{defi}
\subsection{The classical Landau Theorem and Bloch Theorem} The classical Landau theorem states that if $f$ is an analytic function in the unit disk $\mathbb{D}$ with $f(0)=0$, $f'(0)=1$, and $|f(z)|<M$ in $\mathbb{D}$ for some $M>1$, then $f$ is univalent in $\mathbb{D}_{r_0}$ with 
\beas r_0=\frac{1}{M+\sqrt{M^2-1}},\eeas
and $f(\mathbb{D}_{r_0})$ contains a disk $\mathbb{D}_{R_0}$ with $R_0=Mr_0^2$. This result is sharp, with the extremal function $f_0(z)=Mz(1-Mz)/(M-z)$. The Bloch theorem asserts the 
existence of a positive constant number $b$ such that if $f$ is an analytic function in the unit disk $\mathbb{D}$ with $f'(0)=1$, then $f(\mathbb{D})$ contains a schlicht disk of radius $b$, {\it i.e.}, 
a disk of radius $b$ which is the univalent image of some region in $\mathbb{D}$. The supremum of all such constants $b$ is called the Bloch constant (see \cite{CGH2000}).\\[2mm]
In 2000, Chen {\it et al.} \cite{CGH2000} first established the Landau-type theorems for harmonic mappings in the unit disk $\mathbb{D}$. 
\begin{theoA}\cite{CGH2000}
Let $f$ be a harmonic mapping of the unit disk $\mathbb{D}$ such that $f(0)=0$, $f_{\ol{z}}(0)=0$, $f_z(0)=1$, and $|f(z)|<M$ for $z\in\mathbb{D}$. Then, $f$ is univalent on a disk $\mathbb{D}_{r_0}$
with $r_0=\pi^2/(16mM)$ and $f(\mathbb{D}_{r_0})$ contains a schlicht disk $\mathbb{D}_{R_0}$ with $R_0=r_0/2=\pi^2/(32mM)$, where $m\approx 6.85$ is the minimum of the function $(3-r^2)/\left(r(1-r^2)\right)$ for $0<r<1$. 
\end{theoA}
\begin{theoB}\cite{CGH2000}
Let $f$ be a harmonic mapping of the unit disk $\mathbb{D}$ such that $f(0)=0$, $\lambda_f(0)=1$, and $\Lambda_f(z)<\Lambda$ for $z\in\mathbb{D}$. Then, $f$ is univalent on a disk $\mathbb{D}_{r_1}$
with $r_1=\pi/(4(1+\Lambda))$ and $f(\mathbb{D}_{r_1})$ contains a schlicht disk $\mathbb{D}_{R_1}$ with $R_1=r_1/2=\pi/(8(1+\Lambda))$.
\end{theoB}
It is important to note that \textrm{Theorems A} and \textrm{B} are not sharp. 
Better estimates were subsequently provided in \cite{DN2004}, and it has been the subject of further investigation by numerous researchers (see \cite{CG2011, CPW2011, G2006, 
H2008, 1L2009, 2L2009, LC2018, Z2015}). In particular, Liu and Chen \cite{LC2018} have established the sharp version of \textrm{Theorem B} for $\Lambda >1$ using certain 
geometric methods.
The Landau-type theorem for various classes of harmonic mappings has been studied in \cite{1AK2024,2AK2024,3AK2024,BL2019,LL2021,LL2023,ZLK2024}.\\[2mm]
In 2022, Abdulhadi and Hajj \cite{AH2022} were the first to study the Landau-type theorem for poly-analytic functions.
\begin{theoC}\cite{AH2022}
Let $F(z)=\sum_{k=0}^{m-1}\ol{z}^k f_k(z)$ be a poly-analytic function of order $m$ $(\geq2)$ 
in the unit disk $\mathbb{D}$, where $f_k$ are holomorphic in $\mathbb{D}$, satisfying $f_k(0)=0$, $f_k'(0)=1$ and $\left|f_k(z)\right|\leq M$, $M>1$ for all $k$ and $z\in\mathbb{D}$. Then there is a 
constant $0<r_3<1$ so that $F$ is univalent in $|z| <r_3$. In specific, $r_3$ satisfies
\beas 1-M\left(\frac{r_3(2-r_3)}{(1-r_3)^2}+\sum_{k=1}^{m-1}\frac{r_3^k(1+k-kr_3)}{(1-kr_3)^2}\right)=0\eeas
and $F(\mathbb{D}_{r_3})$ contains a disk $\mathbb{D}_{R_3}$, where 
\beas R_3=r_3-r_3^2\left(\frac{1-r_3^{m-1}}{1-r_3}\right)-M\sum_{k=0}^{m-1} \frac{r_3^{k+2}}{1-r_3}.\eeas
\end{theoC}
In consideration of the aforementioned findings, the following questions naturally arise with regard to this study.
\begin{ques}\label{Q1} Can we improve \textrm{Theorem C} further by removing the condition $f_k'(0)=1$ for all $k$?\end{ques}
\begin{ques}\label{Q2} Can we establish the Landau-type theorem for a poly-analytic function of the form (\ref{e1}), where the analytic functions $f_k$ have multiple zeros at the origin?\end{ques}
\begin{ques}\label{Q3}  Can we establish an analogue of \textrm{Theorem C} for reduced poly-analytic functions?\end{ques}
The purpose of this paper is primarily to provide the affirmative answers to Questions \ref{Q1}, \ref{Q2} and \ref{Q3}.
\section{Some lemmas}
To prove our main results, the following lemmas play a crucial role.
\begin{lem}\cite{LP2024}\label{lem3} Let $H(z)$ be an analytic function in $\mathbb{D}$ satisfying $H'(0)=1$ and $|H'(z)|<\Lambda$ for all $z\in\mathbb{D}$ and for some $\Lambda>1$. 
\item[(i)] For all $z_1,z_2\in\mathbb{D}_r$ $(0<r<1, z_1\not=z_2)$, we have
\beas |H(z_1)-H(z_2)|=\left|\int_\gamma H'(z) dz\right|\geq \Lambda\frac{1-\Lambda r}{\Lambda -r}|z_1-z_2|,\eeas
where $\gamma=[z_1,z_2]$ denotes the line segment joining $z_1$ and $z_2$.
\item[(ii)] For $z'\in\pa\mathbb{D}_r$ $(0<r<1)$ with $w'=H(z')\in H(\pa\mathbb{D}_r)$ and $|w'|=\min\{|w|: w\in H(\pa\mathbb{D}_r)\}$, set $\gamma_0=H^{-1}(\Gamma_0)$ and $\Gamma_0=[0, w']$ denotes the line segment joining the origin and $w'$. Then, we have 
\beas |H(z')|\geq \Lambda \int_0^r \frac{1/\Lambda -t}{1-t/\Lambda} dt=\Lambda^2 r+(\Lambda^3-\Lambda)\log(1-r/\Lambda).\eeas 
\end{lem}
\begin{lem}\cite{LP2024}\label{lem1} Suppose that $f(z) =\sum_{n=1}^\infty a_nz^n$ is an analytic function in $\mathbb{D}$ satisfying $|f'(0)|=1$ and $|f(z)|\leq M$ in $\mathbb{D}$ and for some $M > 0$. Then, we have $M\geq 1$ and 
\beas |a_n|\leq M-\frac{1}{M}\quad\text{for}\quad n\geq2.\eeas
These inequalities are sharp, with the extremal functions $f_n(z)$, where
\beas f_1(z)=z, f_n(z)=Mz\frac{1-Mz^{n-1}}{M-z^{n-1}}=z-\left(M-\frac{1}{M}\right)z^n-\sum_{k=3}^\infty \frac{M^2-1}{M^{k-1}}z^{(n-1)(k-1)+1}\eeas
for $n\geq 2$.
\end{lem}
\begin{lem}\label{lem6} Let $F(z)=\sum_{k=0}^{m-1}\ol{z}^k f_k(z)$ be a poly-analytic function of order $m$ 
in the unit disk $\mathbb{D}$, where all the $f_k$ are holomorphic in $\mathbb{D}$. If $f_0(z)=z+\sum_{n=2}^\infty a_nz^n$ and $f_k(z) =\sum_{n=1}^\infty b_{k,n}z^n$ $(k\geq 1)$ are the Taylor series expansions of $f_k$ in $\mathbb{D}$, and satisfy the
condition
\bea\label{f4} \sum_{n=2}^\infty n |a_n|r^{n-1}+\sum_{k=1}^{m-1}r^k\left(\sum_{n=1}^\infty n |b_{k,n}| r^{n-1}\right)+\sum_{k=1}^{m-1} kr^{k}\left(\sum_{n=1}^\infty |b_{k,n}|r^{n-1}\right)< 1\eea
for some $r\in(0, 1)$. Then, $F(z)$ is sense-preserving, univalent and fully starlike in the disk $\mathbb{D}_r$.
\end{lem}
\begin{proof} Following the proof of \textrm{Lemma 3.4} of \cite{LP2024}, we consider 
\beas F(z)=z+\sum_{n=2}^\infty a_nz^n+\sum_{k=1}^{m-1}\ol{z}^k\left(\sum_{n=1}^\infty b_{k,n}z^n\right).\eeas
Therefore,
\beas F_z(z)&=&1+\sum_{n=2}^\infty n a_nz^{n-1}+\sum_{k=1}^{m-1}\ol{z}^k\left(\sum_{n=1}^\infty n b_{k,n}z^{n-1}\right)\quad\text{and}\\[2mm]
F_{\ol{z}}(z)&=&\sum_{k=1}^{m-1} k\ol{z}^{k-1}\left(\sum_{n=1}^\infty b_{k,n}z^n\right).\eeas
It is evident that $J_F(0)=|F_{z}(0)|^2-|F_{\ol{z}}(0)|^2=1$. Using inequality (\ref{f4}), we have
\beas |F_{z}(z)|-|F_{\ol{z}}(z)|&=&\left|1+\sum_{n=2}^\infty n a_nz^{n-1}+\sum_{k=1}^{m-1}\ol{z}^k\left(\sum_{n=1}^\infty n b_{k,n}z^{n-1}\right)\right|\\
&&-\left|\sum_{k=1}^{m-1} k\ol{z}^{k-1}\left(\sum_{n=1}^\infty b_{k,n}z^n\right)\right|\\[2mm]
&> &1-\sum_{n=2}^\infty n |a_n|r^{n-1}-\sum_{k=1}^{m-1}r^k\left(\sum_{n=1}^\infty n |b_{k,n}| r^{n-1}\right)\\[2mm]
&&-\sum_{k=1}^{m-1} kr^{k}\left(\sum_{n=1}^\infty |b_{k,n}|r^{n-1}\right)\geq 0.\eeas
Therefore, $J_F(z)=|F_{z}(z)|^2-|F_{\ol{z}}(z)|^2>0$ for $|z|<r$ and hence, we have $F$ is sense-preserving in $\mathbb{D}_r$. Again, it easy to see that 
\bea\label{f5} zF_z(z)-\ol{z}F_{\ol{z}}(z)-F(z)=\sum_{n=2}^\infty (n-1) a_nz^{n}+\sum_{k=1}^{m-1}\ol{z}^k\left(\sum_{n=1}^\infty n b_{k,n}z^{n}\right)\nonumber\\
-\sum_{k=1}^{m-1}(k+1)\ol{z}^k\left(\sum_{n=1}^\infty b_{k,n}z^n\right).\eea
Fix $r_0\in(0,r]$ and consider the circle $C_{r_0}=\{z\in\mathbb{C}: |z|=r_0\}$. For $z\in C_{r_0}$, it follows from (\ref{f4}) and (\ref{f5}) that
\beas &&\left|zF_z(z)-\ol{z}F_{\ol{z}}(z)-F(z)\right|\\[2mm]
&\leq&\sum_{n=2}^\infty (n-1) |a_n||z|^{n}+\sum_{k=1}^{m-1}|z|^k\left(\sum_{n=1}^\infty n |b_{k,n}||z|^{n}\right)+\sum_{k=1}^{m-1}(k+1)|z|^k\left(\sum_{n=1}^\infty |b_{k,n}| |z|^n\right)\\[1mm]
&=&|z|\left(\sum_{n=2}^\infty n |a_n||z|^{n-1}+\sum_{k=1}^{m-1}|z|^k\left(\sum_{n=1}^\infty n |b_{k,n}| |z|^{n-1}\right)+\sum_{k=1}^{m-1} k |z|^{k}\left(\sum_{n=1}^\infty b_{k,n}|z|^{n-1}\right)\right)\\[1mm]
&&-\sum_{n=2}^\infty|a_n||z|^{n}+\sum_{k=1}^{m-1}|z|^k\left(\sum_{n=1}^\infty |b_{k,n}| |z|^n\right)\\[1mm]
&<& |z|-\sum_{n=2}^\infty|a_n||z|^{n}+\sum_{k=1}^{m-1}|z|^k\left(\sum_{n=1}^\infty |b_{k,n}| |z|^n\right)\\[1mm]
&\leq& \left|f_0(z)+\sum_{k=1}^{m-1}\ol{z}^k f_k(z)\right|=|F(z)|.\eeas
Thus, we have
\beas \left|\frac{zF_z(z)-\ol{z}F_{\ol{z}}(z)}{F(z)}-1\right|<1,\quad{\it i.e.,}\quad
\text{Re}\left(\frac{zF_z(z)-\ol{z}F_{\ol{z}}(z)}{F(z)}\right)>0\quad\text{for}\quad |z|=r_0\eeas
and hence, we have $F$ is univalent on $C_{r_0}$ and it maps $C_{r_0}$ onto a starlike curve. By the sense-preserving property and the degree principle, $F$ is univalent and fully starlike in $\mathbb{D}_r$, since $r_0\in(0, r]$ is arbitrary. This completes the proof.
 \end{proof}
\section{Main results}
In the following result, we obtain the sharp version of the Landau-type theorem for a certain subclass of bounded poly-analytic functions. 
\begin{theo}\label{T1} Suppose that $p_k,m\in\mathbb{N}$ and $M_k\geq 0$ for $k=1,2,\ldots,m-1$. Let $F(z)=\sum_{k=0}^{m-1}\ol{z}^k f_k(z)$ be a poly-analytic function of order $m$ 
in the unit disk $\mathbb{D}$, where all the $f_k$ are holomorphic in $\mathbb{D}$, satisfying $F_z(0)-1=f_0(0)=0$ and $f_k(0)=f_k'(0)=\cdots=f_k^{(p_k-1)}(0)=0$ for $k\in\{1,2,\ldots,m-1\}$. If $|f_0'(z)|<M_0$ for some $M_0>0$
and $\left|f_k^{(p_k)}(z)\right|\leq M_k$ for $k\in\{1,2,\ldots,m-1\}$ and for all $z\in\mathbb{D}$, then $M_0>1$ and $F(z)$ is univalent in $\mathbb{D}_{r_0}$ and $F(\mathbb{D}_{r_0})$ contains 
a schlicht disk $\mathbb{D}_{r_1}$, where $r_0\in(0, 1)$ is the unique positive root of the equation 
\beas &&M_0\frac{1-M_0 r}{M_0-r}-\sum_{k=1}^{m-1}(k+p_k)M_k \frac{r^{k+p_k-1}}{p_k!}=0\\[2mm]\text{and}
&&r_1=M_0^2 r_0+(M_0^3-M_0)\log\left(1-\frac{r_0}{M_0}\right)-\sum_{k=1}^{m-1}M_k \frac{r_0^{p_k+k}}{p_k!}.\eeas
The result is sharp, with an extremal function given by 
\bea\label{f3} G_1(z)=M_0^2 z+(M_0^3-M_0)\log\left(1-\frac{z}{M_0}\right)-\sum_{k=1}^{m-1}M_k \ol{z}^k\frac{z^{p_k}}{p_k!}.\eea
\end{theo}
\begin{proof} First, we prove that $F$ is univalent in $\mathbb{D}_{r_0}$. Let $z_1,z_2(\not=z_1)\in\mathbb{D}_r$ for $0<r<r_0$.
For the line segment $[z_1, z_2]$ joining $z_1$ and $z_2$, we have
\beas |F(z_2)-F(z_1)|&=&\left|\int_{[z_1,z_2]} \left(F_z(z) dz+F_{\ol{z}}(z) d\ol{z}\right)\right|\\[2mm]
&\geq&\left|\int_{[z_1,z_2]} f_0'(z)dz\right|-\left|\int_{[z_1,z_2]} \left(\sum_{k=1}^{m-1}\ol{z}^k f_k'(z) dz+\sum_{k=1}^{m-1}k\ol{z}^{k-1} f_k(z) d\ol{z}\right)\right|.\eeas
Given that $F_z(0)=1$, it follows that $f_0'(0)=1$, and $|f_0'(z)|<M_0$ for all $z\in\mathbb{D}$ implies that $M_0>1$. In view of \textrm{Lemma \ref{lem3}}, we have 
\bea\label{g1} \left|\int_{[z_1,z_2]} f_0'(z) dz\right|\geq M_0\frac{1-M_0 r}{M_0 -r}|z_1-z_2|.\eea
By the hypothesis, we have that $f_k(z)$ are analytic in $\mathbb{D}$ with $f_k(0)=f_k'(0)=\cdots=f_k^{(p_k-1)}(0)=0$ and $\left|f_k^{(p_k)}(z)\right|\leq M_k$ for $k\in\{1,2,\ldots,m-1\}$ and for all $z\in\mathbb{D}$, therefore
\beas \left|f_k^{(p_k-1)}(z)\right|&=&\left|\int_{[0,z]}f_k^{(p_k)}(z)dz\right|\leq \int_{[0,z]} \left|f_k^{(p_k)}(z)\right| |dz|\leq M_k\int_0^1 |z| dt=M_k |z|,\\[2mm]
\left|f_k^{(p_k-2)}(z)\right|&=&\left|\int_{[0,z]}f_k^{(p_k-1)}(z)dz\right|\leq \int_{[0,z]} \left|f_k^{(p_k-1)}(z)\right| |dz|\leq M_k\int_0^1 |z|^2 dt=M_k \frac{|z|^2}{2}.\eeas
As a result of the continuation of this process, we finally have
\bea\label{f1} \left|f_k'(z)\right|\leq M_k \frac{|z|^{p_k-1}}{(p_k-1)!}\quad\text{and}\quad \left|f_k(z)\right|\leq M_k \frac{|z|^{p_k}}{p_k!}\quad\text{for}\quad k=1,2,\ldots,m-1.\eea
Using (\ref{f1}), we have
\beas&&\left|\int_{[z_1,z_2]} \left(\sum_{k=1}^{m-1}\ol{z}^k f_k'(z) dz+\sum_{k=1}^{m-1}k\ol{z}^{k-1} f_k(z) d\ol{z}\right)\right|\\[2mm]
&\leq& \sum_{k=1}^{m-1}\int_{[z_1,z_2]} |\ol{z}|^k |f_k'(z)| |dz|+\sum_{k=1}^{m-1}\int_{[z_1,z_2]} k|\ol{z}|^{k-1} |f_k(z)| |d\ol{z}|\\[2mm]
&\leq& \sum_{k=1}^{m-1}\int_{[z_1,z_2]} r^k M_k \frac{r^{p_k-1}}{(p_k-1)!} |dz|+\sum_{k=1}^{m-1}\int_{[z_1,z_2]} kr^{k-1}M_k \frac{r^{p_k}}{p_k!} |d\ol{z}|\\[2mm]
&=&|z_2-z_1|\sum_{k=1}^{m-1}(k+p_k) M_k \frac{r^{k+p_k-1}}{p_k!}.\eeas
Therefore, we have
\bea\label{f2}  |F(z_2)-F(z_1)|\geq \left(M_0\frac{1-M_0 r}{M_0 -r}-\sum_{k=1}^{m-1}(k+p_k) M_k \frac{r^{k+p_k-1}}{p_k!}\right)|z_1-z_2|.\eea
Let \beas G(r)=M_0\frac{1-M_0 r}{M_0 -r}-\sum_{k=1}^{m-1}(k+p_k) M_k \frac{r^{k+p_k-1}}{p_k!},\quad r\in[0,1].\eeas 
Differentiating $G(r)$ with respect to $r$, we obtain
\beas G'(r)=M_0\frac{1-M_0^2}{(M_0-r)^2}-\sum_{k=1}^{m-1}(k+p_k)(k+p_k-1) M_k \frac{r^{k+p_k-2}}{p_k!}<0\quad\text{for}\quad r\in[0,1].\eeas
Therefore, $G(r)$ is a monotonically decreasing function of $r\in[0,1]$. It is evident that $G(0)=1$ and $G(1)=-M_0-\sum_{k=1}^{m-1}(k+p_k) M_k/p_k!<0$. As the function $G(r)$ is 
continuous on $[0,1]$, in view of the mean value theorem, there exists a unique $r_0\in(0,1)$ such that $G(r_0)=0$. 
From (\ref{f2}), we have 
\beas |F(z_2)-F(z_1)|>0\quad\text{for}\quad r<r_0.\eeas 
Thus, $F(z_1)\not=F(z_2)$, which shows that $F$ is univalent in the disk $\mathbb{D}_{r_0}$.\\[2mm]
Next, we prove that $F(\mathbb{D}_{r_0})$ contains the schlicht disk $\mathbb{D}_{r_1}$. Note that $F(0)=f_0(0)=0$. For $z'\in\pa\mathbb{D}_{r_0}$ with $w'=F(z')\in F(\pa \mathbb{D}_{r_0})$ and $|w'|=\min\{ |w|: w\in F(\pa \mathbb{D}_{r_0})\}$, by \textrm{Lemma \ref{lem3}} and (\ref{f1}), we have 
\beas |w'|&=&|F(z')|=\left|\sum_{k=0}^{m-1}\ol{z'}^k f_k(z') \right|\\[2mm]
&\geq& |f_0(z')|-\left|\sum_{k=1}^{m-1}\ol{z'}^k f_k(z') \right|\\[2mm]
&\geq& M_0^2 r_0+(M_0^3-M_0)\log\left(1-\frac{r_0}{M_0}\right)-\sum_{k=1}^{m-1}M_k \frac{r_0^{p_k+k}}{p_k!}=r_1,\eeas
which implies that $F(\mathbb{D}_{r_0})\supseteq \mathbb{D}_{r_1}$.\\[2mm]
\indent In order to prove the sharpness of the constants $r_0$ and $r_1$, we consider the the poly-analytic function $G_1(z)$ given by (\ref{f3}). It is evident that $G_1(z)$ satisfies all the 
conditions of \textrm{Theorem \ref{T1}}, and thus, we have that $G_1(z)$ is univalent in $\mathbb{D}_{r_0}$ and $G_1(\mathbb{D}_{r_0})\supseteq \mathbb{D}_{r_1}$. For the sharpness of $r_0$, we need 
to prove that $G_1(z)$ is not univalent in $\mathbb{D}_r$ for $r> r_0$. Let 
\beas G_2(x)=M_0^2 x+(M_0^3-M_0)\log\left(1-\frac{x}{M_0}\right)-\sum_{k=1}^{m-1}M_k \frac{x^{k+p_k}}{p_k!},\quad x\in[0,1].\eeas
It is evident that 
\beas G_2'(x)=M_0\frac{1-M_0x}{M_0-x}-\sum_{k=1}^{m-1}(k+p_k)M_k \frac{x^{k+p_k-1}}{p_k!}\eeas
is continuous and monotonically decreasing function of $x\in[0,1]$ with $G_2'(r_0)=0$, where $r_0\in(0,1)$ is unique. Therefore, $G_2(x)$ is strictly increasing on $[0,r_0)$ and 
strictly decreasing on $(r_0,1]$. Since $G_2(0)=0$, there is a unique $r_2\in(r_0,1]$ such that $G_2(r_2)=0$ if $G_2(1)\leq 0$ and $G_2(r_0)>G(0)=0$, {\it i.e.,}
\beas M_0^2 r_0+(M_0^3-M_0)\log\left(1-\frac{r_0}{M_0}\right)-\sum_{k=1}^{m-1}M_k \frac{r_0^{k+p_k}}{p_k!}>0.\eeas
For fixed $r\in(r_0,1]$, let 
\beas \epsilon=\left\{\begin{array}{lll}
\min\{(r-r_0)/2, (r_2-r_0)/2\},&\text{if}\quad G_2(1)\leq 0,\\[2mm]
(r-r_0)/2,&\text{if}\quad G_2(1)>0.
\end{array}\right.\eeas
Let $x_1=r_0+\epsilon$. By the mean value theorem, there exists a unique $\delta\in(0,r_0)$ such that $x_2=r_0-\delta$ and $G_2(x_1)=G_2(x_2)$. Let $z_1=x_1$ and $z_2=x_2$, then it is evident that $z_1,z_2\in\mathbb{D}_r$ with $z_1\not=z_2$ and $G_1(z_1)=G_1(x_1)=G_2(x_1)=G_2(x_2)=G_1(z_2)$. This shows that $G_1(z)$ is not univalent in $\mathbb{D}_r$, where $r\in(r_0,1]$ and thus, the radius $r_0$ is sharp.\\[2mm]
To prove the sharpness of $r_1$, let $z'=r_0\in\pa\mathbb{D}_{r_0}$. Then, we have
\beas |G_1(z')-G_1(0)|=|G_1(r_0)|=|G_2(r_0)|=G_2(r_0)=r_1.\eeas
Hence, the radius $r_1$ of the schlicht disk is also sharp. This completes the proof. 
\end{proof}
\begin{rem} Setting $m=2$, $p_1=1$ in \textrm{Theorem \ref{T1}} gives \textrm{Theorem 2.1} of \cite{LP2024}.\end{rem}
\begin{rem} Setting $p_k=1$ for $k=1,2,\ldots,m-1$ in \textrm{Theorem \ref{T1}} gives \textrm{Theorem 3.3} of \cite{4AK2024}.\end{rem}
In the following result, we establish an improved version of \textrm{Theorem C} without the conditions $f_k'(0)=1$ for $k=1,2,\ldots, m-1$.
\begin{theo}\label{T2} Suppose that $m\in\mathbb{N}$ and $M_k\geq 0$ for $k=0,1,2,\ldots,m-1$. Let $F(z)=\sum_{k=0}^{m-1}\ol{z}^k f_k(z)$ be a poly-analytic function of order $m$ 
in the unit disk $\mathbb{D}$, where all the $f_k$ are holomorphic in $\mathbb{D}$, satisfying $F_z(0)-1=f_k(0)=0$ for $k\in\{0,1,2,\ldots,m-1\}$. If $\left|f_k(z)\right|\leq M_k$ for $k\in\{0,1,2,\ldots,m-1\}$ and for all $z\in\mathbb{D}$, then $M_0\geq 1$ and $F(z)$ is univalent in $\mathbb{D}_{r_2}$ and $F(\mathbb{D}_{r_2})$ contains 
a schlicht disk $\mathbb{D}_{r_3}$, where $r_2\in(0, 1)$ is the unique positive root of the equation 
\beas &&1-\left(M_0-\frac{1}{M_0}\right)\frac{(2-r)r}{(1-r)^2}-\sum_{k=1}^{m-1}r^k \frac{M_k}{1-r^2}-\sum_{k=1}^{m-1}kr^{k}M_k=0\\[2mm]\text{and}
&&r_3=r_2-\left(M_0-\frac{1}{M_0}\right)\frac{r_2^2}{1-r_2}-\sum_{k=1}^{m-1}M_k r_2^k.\eeas
\end{theo}
\begin{proof} First, we prove that $F$ is univalent in $\mathbb{D}_{r_2}$. Let $z_1,z_2(\not=z_1)\in\mathbb{D}_r$ for $0<r<r_2$. Since $F_z(0)=1$, it follows that $f_0'(0)=1$.
For the line segment $[z_1, z_2]$ joining $z_1$ and $z_2$, we have
\beas |F(z_2)-F(z_1)|&=&\left|\int_{[z_1,z_2]} \left(F_z(z) dz+F_{\ol{z}}(z) d\ol{z}\right)\right|\\
&=&\left|\int_{[z_1,z_2]} \left(F_z(0) dz+F_{\ol{z}}(0) d\ol{z}\right)\right.\\[1mm]&&\left.+\int_{[z_1,z_2]} \left((F_z(z)-F_z(0))dz+(F_{\ol{z}}(z)-F_{\ol{z}}(0)) d\ol{z}\right)\right|\eeas
\beas\qquad&\geq &\left|\int_{[z_1,z_2]} dz\right|-\left|\int_{[z_1,z_2]} \left(f_0'(z)-1\right)dz\right|\\[2mm]
&&-\left|\int_{[z_1,z_2]} \left(\sum_{k=1}^{m-1}\ol{z}^k f_k'(z) dz+\sum_{k=1}^{m-1}k\ol{z}^{k-1} f_k(z) d\ol{z}\right)\right|.\eeas
Since $f_0(0)=0$, $f_0'(0)=1$ and $|f_0(z)|\leq M_0$ for $z\in\mathbb{D}$, let $f_0(z)=z+\sum_{n=2}^\infty a_{n,0}z^n$ be the Taylor series expansion of $f_0(z)$.  
In view of \textrm{Lemma \ref{lem1}}, we have $M_0\geq 1$ and $|a_{n,0}|\leq M_0-1/M_0$ for $n\geq 2$. Thus, we have 
\bea\label{g5} \left|\int_{[z_1,z_2]}\left( f_0'(z)-1\right) dz\right|&\leq&  |z_1-z_2|\left(M_0-\frac{1}{M_0}\right)\sum_{n=2}^\infty nr^{n-1}\nonumber\\[2mm]
&=& |z_1-z_2|\left(M_0-\frac{1}{M_0}\right)\frac{(2-r)r}{(1-r)^2}.\eea
By the hypothesis, we have that $f_k(z)$ are analytic in $\mathbb{D}$ with $f_k(0)=0$ and $\left|f_k(z)\right|\leq M_k$ in $\mathbb{D}$ for $k\in\{1,2,\ldots,m-1\}$.
According to the Schwarz's lemma, if $f:\mathbb{D}\to\mathbb{D}$ is analytic with $f(0)=0$, then $|f(z)|\leq |z|$.  A sharpened form of the Schwarz lemma, known as the Schwarz-Pick lemma, gives the estimate $(1-|z|^2)|f'(z)|\leq 1-|f(z)|^2\leq 1$ for any analytic function $f:\mathbb{D}\to\mathbb{D}$ and $z\in\mathbb{D}$.
Thus, we have
\bea\label{f7}\left|f_k(z)\right|\leq M_k |z|\quad\text{and}\quad \left|f_k'(z)\right|\leq\frac{M_k}{1-r^2}\quad\text{for}\quad k=1,2,\ldots,m-1.\eea
Using (\ref{f7}), we have
\beas&&\left|\int_{[z_1,z_2]} \left(\sum_{k=1}^{m-1}\ol{z}^k f_k'(z) dz+\sum_{k=1}^{m-1}k\ol{z}^{k-1} f_k(z) d\ol{z}\right)\right|\\[2mm]
&\leq& \sum_{k=1}^{m-1}\int_{[z_1,z_2]} |\ol{z}|^k |f_k'(z)| |dz|+\sum_{k=1}^{m-1}\int_{[z_1,z_2]} k|\ol{z}|^{k-1} |f_k(z)| |d\ol{z}|\\[2mm]
&\leq& \sum_{k=1}^{m-1}\int_{[z_1,z_2]} r^k\frac{M_k}{1-r^2} |dz|+\sum_{k=1}^{m-1}\int_{[z_1,z_2]} kr^{k}M_k|d\ol{z}|\\[2mm]
&=&|z_2-z_1|\left(\sum_{k=1}^{m-1}r^k \frac{M_k}{1-r^2}+\sum_{k=1}^{m-1}kr^{k}M_k\right).\eeas
Therefore, 
\be\label{f6}  |F(z_2)-F(z_1)|\geq \left(1-\left(M_0-\frac{1}{M_0}\right)\frac{(2-r)r}{(1-r)^2}-\sum_{k=1}^{m-1}r^k \frac{M_k}{1-r^2}-\sum_{k=1}^{m-1}kr^{k}M_k\right)|z_1-z_2|.\ee
Let \beas G_3(r)=1-\left(M_0-\frac{1}{M_0}\right)\frac{(2-r)r}{(1-r)^2}-\sum_{k=1}^{m-1}r^k\frac{ M_k}{(1-r^2)}-\sum_{k=1}^{m-1}kr^{k}M_k,\quad r\in[0,1).\eeas 
Differentiating $G_3(r)$ with respect to $r$, we obtain
\beas G_3'(r)=-\frac{2\left(M_0-1/M_0\right)}{(1-r)^3}-\sum_{k=1}^{m-1} M_k\left(\frac{2 r^{k+1}}{(1-r^2)^2}+\frac{k r^{k-1}}{1-r^2}\right)-\sum_{k=1}^{m-1}k^2r^{k-1}M_k<0\eeas
for $r\in[0,1)$.
Therefore, $G_3(r)$ is a monotonically decreasing function of $r\in[0,1)$ with $G_3(0)=1$ and $\lim_{r\to 1^-}G_3(r)=-\infty$. As the function $G_3(r)$ is continuous on $[0,1)$, in view of the mean value theorem, there exists a unique $r_2\in(0,1)$ such that $G_3(r_2)=0$. 
From (\ref{f6}), we have 
\beas |F(z_2)-F(z_1)|>0\quad\text{for}\quad r<r_2.\eeas 
Thus, $F(z_1)\not=F(z_2)$, which shows that $F$ is univalent in the disk $\mathbb{D}_{r_2}$.\\[2mm]
Next, we prove that $F(\mathbb{D}_{r_2})$ contains the schlicht disk $\mathbb{D}_{r_3}$. Note that $F(0)=f_0(0)=0$. For any $z'\in\pa\mathbb{D}_{r_2}$ with $w'=F(z')\in F(\pa \mathbb{D}_{r_2})$, it follows from \textrm{Lemma \ref{lem1}} that 
\beas |w'|=|F(z')|&=&\left|z'+\sum_{n=2}^\infty a_{n,0}z'^n+\sum_{k=1}^{m-1}\ol{z'}^k f_k(z') \right|\\[2mm]
&\geq&\left|z'\right|-\sum_{n=2}^\infty \left|a_{n,0}\right|\left|z'\right|^n-\sum_{k=1}^{m-1}\left|\ol{z'}\right|^k \left|f_k(z')\right|\\[2mm]
&\geq& r_2-\left(M_0-\frac{1}{M_0}\right)\frac{r_2^2}{1-r_2}-\sum_{k=1}^{m-1}M_k r_2^k=r_3,\eeas
which implies that $F(\mathbb{D}_{r_2})\supseteq \mathbb{D}_{r_3}$. This completes the proof.
\end{proof}
\begin{rem} 
If $M_k=M$ for $k\in\{0,1,\ldots,m-1\}$ in \textrm{Theorem \ref{T2}}, then the result is an improved version of \textrm{Theorem C}, in which the conditions $f_k'(0)=1$ are removed.
\end{rem}
The addition of the conditions $f_k'(0)=1$ for $k=1,2,…,m-1$ to \textrm{Theorem \ref{T2}} yields the following result.
 \begin{theo}\label{T3} Suppose that $m\in\mathbb{N}$ and $M_k\geq 0$ for $k=0,1,\ldots,m-1$. Let $F(z)=\sum_{k=0}^{m-1}\ol{z}^k f_k(z)$ be a poly-analytic function of order $m$ 
in the unit disk $\mathbb{D}$, where all the $f_k$ are holomorphic in $\mathbb{D}$, satisfying $F_z(0)-1=f_k(0)=0$ for $k\in\{0,1,2,\ldots,m-1\}$ and $f_k'(0)=1$ $(k=1,2,\dots,m-1)$. If $\left|f_k(z)\right|\leq M_k$ for $k\in\{0,1,2,\ldots,m-1\}$ and for all $z\in\mathbb{D}$, then $M_k\geq 1$ for $k\in\{0,1,2,\ldots,m-1\}$ and $F(z)$ is sense-preserving, univalent and fully starlike in the disk $\mathbb{D}_{r_4}$, where $r_4\in(0, 1)$ is the unique positive root of the equation 
\beas 1-\sum_{k=0}^{m-1}\left(M_k-\frac{1}{M_k}\right)\frac{r^{k+1}(2-r+k(1-r))}{(1-r)^2}-\sum_{k=1}^{m-1}(k+1)r^k=0.\eeas
\end{theo}
\begin{proof} 
Let $f_0(z)=z+\sum_{n=2}^\infty a_nz^n$ and $f_k(z) =z+\sum_{n=2}^\infty b_{k,n}z^n$ for $k=1,2,\ldots, m-1$.
In view of \textrm{Lemma \ref{lem1}}, we have $M_k\geq 1$ for $k\in\{0,1,2,\ldots,m-1\}$, and 
\beas |a_n|\leq M_0-\frac{1}{M_0},~ |b_{k,n}|\leq M_k-\frac{1}{M_k}\quad\text{for}\quad n\geq 2\quad\text{and}\quad k=1,2,\ldots,m-1.\eeas 
Therefore, 
\beas &&\sum_{n=2}^\infty n |a_n|r^{n-1}+\sum_{k=1}^{m-1}r^k\left(\sum_{n=1}^\infty n |b_{k,n}| r^{n-1}\right)+\sum_{k=1}^{m-1} kr^{k}\left(\sum_{n=1}^\infty |b_{k,n}|r^{n-1}\right)\\[2mm]
 &=&\sum_{n=2}^\infty n |a_n|r^{n-1}+\sum_{k=1}^{m-1}r^k\left(1+\sum_{n=2}^\infty n |b_{k,n}| r^{n-1}\right)+\sum_{k=1}^{m-1} kr^{k}\left(1+\sum_{n=2}^\infty |b_{k,n}|r^{n-1}\right)\\[2mm]
&\leq& \left(M_0-\frac{1}{M_0}\right)\sum_{n=2}^\infty n r^{n-1}+\sum_{k=1}^{m-1}\left(M_k-\frac{1}{M_k}\right)r^k\left(\sum_{n=2}^\infty n r^{n-1}\right)\\[2mm]
&&+\sum_{k=1}^{m-1} k\left(M_k-\frac{1}{M_k}\right)r^{k}\left(\sum_{n=2}^\infty r^{n-1}\right)+\sum_{k=1}^{m-1}(k+1)r^k\\[2mm]
&=&\sum_{k=0}^{m-1}\left(M_k-\frac{1}{M_k}\right)\frac{r^{k+1}(2-r+k(1-r))}{(1-r)^2}+\sum_{k=1}^{m-1}(k+1)r^k< 1\eeas
for $r<r_4$, where $r_4\in(0, 1)$ is the unique positive root of the equation 
\beas H(r):=1-\sum_{k=0}^{m-1}\left(M_k-\frac{1}{M_k}\right)\frac{r^{k+1}(2-r+k(1-r))}{(1-r)^2}-\sum_{k=1}^{m-1}(k+1)r^k=0.\eeas
It is evident that 
\beas H'(r)&=&-\sum_{k=0}^{m-1}\left(M_k-\frac{1}{M_k}\right)\frac{r^{k} \left(k^2 (1-r)^2+k (3-r) (1-r)+2\right)}{(1-r)^3}\\[1mm]
&&-\sum_{k=1}^{m-1} k (k+1)r^{k-1}<0\quad\text{for}\quad r\in[0, 1),\eeas
which shows that $H(r)$ is a monotonically decreasing function of $r\in[0,1)$ with $H(0)=1$ and $\lim_{r\to1^-}H(r)=-\infty$. As the function $H(r)$ is 
continuous on $[0,1)$, in view of the mean value theorem, there exists a unique $r_4\in(0,1)$ such that $H(r_4)=0$.
Therefore, in view of \textrm{Lemma \ref{lem6}}, we have $F(z)$ is sense-preserving, univalent and fully starlike in the disk $\mathbb{D}_{r_4}$, where $r_4\in(0, 1)$ is the unique positive root of the equation $H(r)=0$. This completes the proof.
 \end{proof}
In the following result, we establish the Landau-type theorem for certain subclass of reduced poly-analytic functions.
 \begin{theo}\label{T4} Suppose that $m\in\mathbb{N}$ and $M_k\geq 0$ for $k=1,2,\ldots,m-1$. Let $F(z)=\sum_{k=0}^{m-1}|z|^{2k} f_k(z)$ be a reduced poly-analytic function of order $m$ 
in the unit disk $\mathbb{D}$, where all the $f_k$ are holomorphic in $\mathbb{D}$, satisfying $F_z(0)-1=f_k(0)=0$ for $k\in\{0,1,2,\ldots,m-1\}$. If $|f_0'(z)|<M_0$ for some $M_0>0$
and $\left|f_k'(z)\right|\leq M_k$ for $k\in\{1,2,\ldots,m-1\}$ and for all $z\in\mathbb{D}$, then $M_0>1$ and $F(z)$ is univalent in $\mathbb{D}_{r_5}$ and $F(\mathbb{D}_{r_5})$ contains 
a schlicht disk $\mathbb{D}_{r_6}$, where $r_5\in(0, 1)$ is the unique positive root of the equation 
\beas &&M_0\frac{1-M_0 r}{M_0 -r}-\sum_{k=1}^{m-1}M_k\left(2k+1\right)  r^{2k}=0\eeas
\beas \text{and}&&r_6=M_0^2 r_5+(M_0^3-M_0)\log\left(1-\frac{r_5}{M_0}\right)-\sum_{k=1}^{m-1}M_k r_5^{2k+1}.\eeas
The result is sharp, with an extremal function given by 
\bea\label{f9} G_4(z)=M_0^2 z+(M_0^3-M_0)\log\left(1-\frac{z}{M_0}\right)-\sum_{k=1}^{m-1}|z|^{2k}M_k z.\eea
\end{theo}
\begin{proof}Given that $F_z(0)=1$, it follows that $f_0'(0)=1$, and $|f_0'(z)|<M_0$ for all $z\in\mathbb{D}$ implies that $M_0>1$. In view of \textrm{Lemma \ref{lem3}}, we have the inequality (\ref{g1}). By the hypothesis, we have that $f_k(z)$ is analytic in $\mathbb{D}$ with $f_k(0)=0$ and $\left|f_k'(z)\right|\leq M_k$ in $\mathbb{D}$ for $k\in\{1,2,\ldots,m-1\}$, therefore
\bea\label{g2} \left|f_k(z)\right|&=&\left|\int_{[0,z]}f_k'(z)dz\right|\leq \int_{[0,z]} \left|f_k'(z)\right| |dz|\leq M_k\int_0^1 |z| dt=M_k |z|.\eea
First, we prove that $F$ is univalent in $\mathbb{D}_{r_5}$. Let $z_1,z_2(\not=z_1)\in\mathbb{D}_r$ for $0<r<r_5$.
For the line segment $[z_1, z_2]$ joining $z_1$ and $z_2$, we have
\beas &&|F(z_2)-F(z_1)|\\[2mm]
&=&\left|\int_{[z_1,z_2]} \left(F_z(z) dz+F_{\ol{z}}(z) d\ol{z}\right)\right|\\[2mm]
&=&\left|\int_{[z_1,z_2]} \left(\sum_{k=0}^{m-1}\ol{z}^k\left(kz^{k-1}f_k(z)+ z^k f_k'(z)\right) dz+\sum_{k=0}^{m-1}k\ol{z}^{k-1}z^k f_k(z) d\ol{z}\right)\right|\\[2mm]
&\geq&\left|\int_{[z_1,z_2]} f_0'(z)dz\right|-\left|\int_{[z_1,z_2]} \left(\sum_{k=1}^{m-1}\ol{z}^k \left(kz^{k-1}f_k(z)+ z^k f_k'(z)\right) dz\right.\right.\\[2mm]&&\left.\left.+\sum_{k=1}^{m-1}k\ol{z}^{k-1} z^kf_k(z) d\ol{z}\right)\right|.\eeas
Using (\ref{g2}), we have
\beas&&\left|\int_{[z_1,z_2]} \left(\sum_{k=1}^{m-1}\ol{z}^k \left(kz^{k-1}f_k(z)+ z^k f_k'(z)\right) dz+\sum_{k=1}^{m-1}k\ol{z}^{k-1} z^k f_k(z) d\ol{z}\right)\right|\\[2mm]
&\leq& \sum_{k=1}^{m-1}\int_{[z_1,z_2]} |\ol{z}|^k \left(k|z|^{k-1}|f_k(z)|+ |z|^k |f_k'(z)|\right) |dz|+\sum_{k=1}^{m-1}\int_{[z_1,z_2]} k|\ol{z}|^{k-1}|z|^k |f_k(z)| |d\ol{z}|\\[2mm]
&\leq& \sum_{k=1}^{m-1}\int_{[z_1,z_2]}\left(k +1\right)r^{2k} M_k  |dz|+\sum_{k=1}^{m-1}\int_{[z_1,z_2]} k M_k r^{2k} |d\ol{z}|\\[2mm]
&=&|z_2-z_1|\sum_{k=1}^{m-1}M_k\left(2k+1\right)  r^{2k}.\eeas
Therefore, we have
\bea\label{g3}  |F(z_2)-F(z_1)|\geq \left(M_0\frac{1-M_0 r}{M_0 -r}-\sum_{k=1}^{m-1}M_k\left(2k+1\right)  r^{2k}\right)|z_1-z_2|.\eea
Let $G_5(r)=M_0(1-M_0 r)/(M_0 -r)-\sum_{k=1}^{m-1}M_k\left(2k+1\right) r^{2k}$ for $r\in[0,1]$. 
It is evident that 
\beas G_5'(r)=\frac{M_0\left(1-M_0^2\right)}{(M_0-r)^2}-\sum_{k=1}^{m-1} 2k M_k\left(2k+1\right) r^{2k-1}<0\quad \text{for}\quad r\in[0, 1],\eeas
which shows that $G_5(r)$ is a monotonically 
decreasing function of $r\in[0,1]$ with $G_5(0)=1$ and $G_5(1)=-M_0-\sum_{k=1}^{m-1}M_k\left(2k+1\right)<0$. As the function $G_5(r)$ is 
continuous on $[0,1]$, in view of the mean value theorem, there exists a unique $r_5\in(0,1)$ such that $G_5(r_5)=0$. From (\ref{g3}), we have 
\beas |F(z_2)-F(z_1)|>0\quad\text{for}\quad r<r_5.\eeas 
Thus, $F(z_1)\not=F(z_2)$, which shows that $F$ is univalent in the disk $\mathbb{D}_{r_5}$.\\[2mm]
Next, we prove that $F(\mathbb{D}_{r_5})$ contains the schlicht disk $\mathbb{D}_{r_6}$. Note that $F(0)=f_0(0)=0$. For $z'\in\pa\mathbb{D}_{r_5}$ with $w'=F(z')\in F(\pa \mathbb{D}_{r_5})$ and $|w'|=\min\{ |w|: w\in F(\pa \mathbb{D}_{r_5})\}$, by \textrm{Lemma \ref{lem3}} and (\ref{g2}), we have 
\beas |w'|&=&|F(z')|=\left|\sum_{k=0}^{m-1}|z'|^{2k} f_k(z') \right|\\[2mm]
&\geq& |f_0(z')|-\left|\sum_{k=1}^{m-1}|z'|^{2k} f_k(z') \right|\\[2mm]
&\geq& M_0^2 r_5+(M_0^3-M_0)\log\left(1-\frac{r_5}{M_0}\right)-\sum_{k=1}^{m-1}M_k r_5^{2k+1}=r_6,\eeas
which implies that $F(\mathbb{D}_{r_5})\supseteq \mathbb{D}_{r_6}$.\\[2mm]
\indent Using similar argument as in the proof of \textrm{Theorem \ref{T1}} with the reduced poly-analytic function $G_4(z)$ given by (\ref{f9}), we can show that the constants $r_5$ and $r_6$ are sharp. This completes the proof.
 \end{proof}
 The following result provides an analogue of \textrm{Theorem C} for reduced poly-analytic function.
\begin{theo}\label{T5} Suppose that $m\in\mathbb{N}$ and $M_k\geq 0$ for $k=0,1,\ldots,m-1$. Let $F(z)=\sum_{k=0}^{m-1}|z|^{2k} f_k(z)$ be a reduced poly-analytic function of order $m$ 
in the unit disk $\mathbb{D}$, where all the $f_k$ are holomorphic in $\mathbb{D}$, satisfying $F_z(0)-1=f_k(0)=0$ for $k\in\{0,1,\ldots,m-1\}$ and $f_k'(0)=1$ $(k=1,2,\dots,m-1)$. If $\left|f_k(z)\right|\leq M_k$ for $k\in\{0,1,\ldots,m-1\}$ and for $z\in\mathbb{D}$, then $M_k\geq 1$ for $k\in\{0,1,\ldots,m-1\}$ and $F(z)$ is univalent in the disk $\mathbb{D}_{r_7}$ and $F(\mathbb{D}_{r_7})$ contains 
a schlicht disk $\mathbb{D}_{r_8}$, where $r_7\in(0, 1)$ is the unique positive root of the equation 
\beas &&1-\sum_{k=0}^{m-1}\left(M_k-\frac{1}{M_k}\right)\frac{\left(2k(1-r)+(2-r)\right)r^{2k+1}}{(1-r)^2}-\sum_{k=1}^{m-1}(2k+1)r^{2k}=0\\\text{and}
&& r_8=r_7-\frac{r_7(r_7^2-r_7^{2m})}{(1-r_7^2)}-\sum_{k=0}^{m-1}\left(M_k-\frac{1}{M_k}\right)\frac{r_7^{2k+2}}{1-r_7}.\eeas
\end{theo}
\begin{proof} 
Given that $F_z(0)=1$, it follows that $f_0'(0)=1$. Thus, we have that all $f_k$ are analytic in $\mathbb{D}$ with $f_k(0)=f_k'(0)-1=0$ for all $k=0,1,2,\ldots, m-1$. Let
$f_k(z) =z+\sum_{n=2}^\infty b_{k,n}z^n$ for all $k$. In view of \textrm{Lemma \ref{lem1}}, we have $M_k\geq 1$ and 
\bea\label{g6} |b_{k,n}|\leq M_k-\frac{1}{M_k}\quad\text{for}\quad n\geq2\;\;\text{and for all $k$}.\eea
First, we prove that $F$ is univalent in $\mathbb{D}_{r_7}$. Let $z_1,z_2(\not=z_1)\in\mathbb{D}_r$ for $0<r<r_7$.
Using similar argument as in the proof of \textrm{Theorem \ref{T2}}, we have 
\bea\label{g8}&& |F(z_2)-F(z_1)|\nonumber\\
&\geq &\left|\int_{[z_1,z_2]} dz\right|-\left|\int_{[z_1,z_2]} \left(f_0'(z)-1\right)dz\right|\\[2mm]
&&-\left|\int_{[z_1,z_2]} \left(\sum_{k=1}^{m-1}\ol{z}^k \left(kz^{k-1}f_k(z)+ z^k f_k'(z)\right)dz+\sum_{k=1}^{m-1}k\ol{z}^{k-1} z^k f_k(z) d\ol{z}\right)\right|.\nonumber\eea
In view of \textrm{Lemma \ref{lem1}}, we obtain the inequality (\ref{g5}).
Using (\ref{g6}), we have
\beas &&\left|\int_{[z_1,z_2]} \left(\sum_{k=1}^{m-1}\ol{z}^k \left(kz^{k-1}f_k(z)+ z^k f_k'(z)\right)dz+\sum_{k=1}^{m-1}k\ol{z}^{k-1}z^k f_k(z) d\ol{z}\right)\right| \\[2mm]
&\leq&\int_{[z_1,z_2]} \left(\sum_{k=1}^{m-1}r^k \left((k+1)r^{k}+\sum_{n=2}^\infty (k+n)|b_{k,n}|r^{n+k-1}\right)|dz|\right.\\&&\left.+\sum_{k=1}^{m-1}\left(kr^{2k} +\sum_{n=2}^\infty k|b_{k,n}|r^{n+2k-1}\right) |d\ol{z}|\right)\\[2mm]
&\leq &|z_1-z_2|\left(\sum_{k=1}^{m-1}r^{2k}\left(M_k-\frac{1}{M_k}\right)\left(\sum_{n=2}^\infty (2k+n)r^{n-1}\right)+\sum_{k=1}^{m-1}(2k+1)r^{2k}\right)\\[2mm]
&=&|z_1-z_2|\left(\sum_{k=1}^{m-1}\left(M_k-\frac{1}{M_k}\right)\frac{\left(2k(1-r)+(2-r)\right)r^{2k+1}}{(1-r)^2}+\sum_{k=1}^{m-1}(2k+1)r^{2k}\right).
\eeas
Using (\ref{g5}) in (\ref{g8}), we have
\bea\label{g7}&&|F(z_2)-F(z_1)|\\
&\geq& \left(1-\sum_{k=0}^{m-1}\left(M_k-\frac{1}{M_k}\right)\frac{\left(2k(1-r)+(2-r)\right)r^{2k+1}}{(1-r)^2}-\sum_{k=1}^{m-1}(2k+1)r^{2k}\right)|z_1-z_2|.\nonumber \eea
Let 
\beas G_6(r)=1-\sum_{k=0}^{m-1}\left(M_k-\frac{1}{M_k}\right)\frac{\left(2k(1-r)+(2-r)\right)r^{2k+1}}{(1-r)^2}-\sum_{k=1}^{m-1}(2k+1)r^{2k}\eeas 
for $r\in[0,1)$. 
It is evident that 
\beas G_6'(r)&=&-\sum_{k=0}^{m-1}\left(M_k-\frac{1}{M_k}\right)\frac{2 r^{2 k} \left(2 k^2 (1-r)^2+k (3-r) (1-r)+1\right)}{(1-r)^3}\\[1mm]
&&-\sum_{k=1}^{m-1} 2k (2k+1)r^{2k-1}<0\quad\text{for}\quad r\in[0, 1),\eeas
which shows that $G_6(r)$ is a monotonically decreasing function of $r\in[0,1)$ with $G_6(0)=1$ and $\lim_{r\to1^-}G_6(r)=-\infty$. As the function $G_6(r)$ is 
continuous on $[0,1)$, in view of the mean value theorem, there exists a unique $r_7\in(0,1)$ such that $G_6(r_7)=0$. From (\ref{g7}), we have 
\beas |F(z_2)-F(z_1)|>0\quad\text{for}\quad r<r_7.\eeas 
Thus, $F(z_1)\not=F(z_2)$, which shows that $F$ is univalent in the disk $\mathbb{D}_{r_7}$.\\[2mm]
Next, we prove that $F(\mathbb{D}_{r_7})$ contains the schlicht disk $\mathbb{D}_{r_8}$. Note that $F(0)=f_0(0)=0$. For any $z'\in\pa\mathbb{D}_{r_7}$ with $w'=F(z')\in F(\pa \mathbb{D}_{r_7})$, it follows from (\ref{g6}) that
\beas |w'|&=&|F(z')|=\left|z'+\sum_{n=2}^\infty b_{0,n}z'^n+\sum_{k=1}^{m-1}|z'|^{2k} \left(z'+\sum_{n=2}^\infty b_{k,n}z'^n\right) \right|\\[2mm]
&\geq&r_7-\sum_{k=1}^{m-1}r_7^{2k+1}-\sum_{k=0}^{m-1}r_7^{2k}\left(M_k-\frac{1}{M_k}\right)\left(\sum_{n=2}^\infty r_7^n\right)\\[2mm]
&=& r_7-\frac{r_7(r_7^2-r_7^{2m})}{(1-r_7^2)}-\sum_{k=0}^{m-1}\left(M_k-\frac{1}{M_k}\right)\frac{r_7^{2k+2}}{1-r_7}=r_8,\eeas
which implies that $F(\mathbb{D}_{r_7})\supseteq \mathbb{D}_{r_8}$. This completes the proof.
\end{proof}
\section*{Declarations}
\noindent{\bf Acknowledgment:} The work of the second author is supported by University Grants Commission (IN) fellowship (No. F. 44 - 1/2018 (SA - III)).\\[1mm]
{\bf Conflict of Interest:} The authors declare that there are no conflicts of interest regarding the publication of this paper.\\[1mm]
{\bf Availability of data and materials:} Not applicable.\\[1mm]
{\bf Authors' contributions:} All authors contributed equally to the investigation of the problem, and all authors have read and approved the final manuscript.

\end{document}